\documentclass[11pt, oneside,reqno]{amsart}
\usepackage{geometry}                
\usepackage{graphicx}
\usepackage{subfig}
\usepackage{amssymb}
\usepackage{epstopdf}
\usepackage{mathrsfs,amssymb}
\usepackage{graphicx}
\usepackage{amscd}
\usepackage{amsmath}
\usepackage{amssymb}
\usepackage{amsthm}
\usepackage{amsfonts}
\usepackage{fancyhdr}
\usepackage{amsxtra}
\usepackage{tikz}
\usepackage{setspace}
\usepackage{enumitem}
\usepackage{amsrefs}
\newtheorem{theorem}{Theorem}
\newtheorem{lemma}{Lemma}

\newtheorem{proposition}{Proposition}

\newtheorem{remark}{Remark}

\newtheorem{definition}{Definition}

\def\bt{\begin{theorem}}
\def\et{\end{theorem}}
\def\bp{\begin{proposition}}
\def\ep{\end{proposition}}
\def\bl{\begin{lemma}}
\def\el{\end{lemma}}

\numberwithin{equation}{section}
\numberwithin{theorem}{section}
\numberwithin{proposition}{section}
\numberwithin{lemma}{section}
\numberwithin{corollary}{section}
\numberwithin{definition}{section}
\numberwithin{claim}{section}

\begin{document}
\setcounter{tocdepth}{1}

\title[emergence of NS-$\alpha$ model for channel flows]{On the emergence of the Navier-Stokes-$\alpha$ model for turbulent channel flows} 

\author{C. Foias$^{1}$}
\author{J. Tian$^{1}$}
\address{$^1$Department of Mathematics\\
Texas A\&M University\\ College Statin, TX, 77843}
\author{B. Zhang$^{1}$}
\address{$\dagger$ corresponding author}
\email[C. Foias]{foias@math.tamu.edu}
\email[J. Tian]{jtian@math.tamu.edu}
\email[B. Zhang$^\dagger$]{bszhang@math.tamu.edu}

\begin{abstract}
In a series of papers (see \cite{CDT02} and the pertinent references therein) the 3D Navier-Stokes-$\alpha$ model were shown to be a useful complement to the 3D Navier-Stokes equations; and in particular, to be a good Reynolds version of the latter equations. In this work, we introduce a simple Reynolds averaging which, due to the wall roughness, transforms the Navier-Stokes equations into the Navier-Stokes-$\alpha$ model.
\end{abstract}

\maketitle
\tableofcontents
\section{Introduction}
It is an accepted fact that the Navier-Stokes equations (NSE) are a good mathematical model for the dynamics of non-turbulent viscous incompressible fluid flows. In this paper, we accept, as done in \cite{A02} and \cite{A04}, that the Navier-Stokes-$\alpha$ model (NS-$\alpha$) are a good mathematical model for the dynamics of appropriately averaged turbulent fluid flows. We will first give our rationale for this acceptance. The possibility that the NS-$\alpha$ are an averaged version of the NSE, first considered in \cite{CFHO98} and \cite{CFHO99}, was entailed by several auspicious facts. Namely, the NS-$\alpha$ analogue of the Poiseuille, resp, Hagen, solution in a channel, resp, a pipe, displays both the classical von K$\acute{a}$rm$\acute{a}$n  and the recent Barenblatt-Chorin laws (\cite{H99}); moreover, the NS-$\alpha$ analogue of the Hagen solution, when suitably calibrated, yields good approximations to many experimental data \cite{CFHO99}. Furthermore, D. D. Holm found an original, physically sound, statistical averaging of the Euler equations such that the NS-$\alpha$ results from the addition of a linear viscous effect (\cite{H99}); see also \cite{CFHOTW99} for a succinct presentation of Holm's averaging. This new averaging is as general as the Reynolds averaging, but unlike the latter, it yields ``closed systems of differential equations". We close the above, rather long, argument by quoting the remarkably successful extension of the classic Blasius theory for a turbulent boundary layer (\cite{A04}) to a larger range of Reynolds numbers, done by using the NS-$\alpha$ instead of the NSE.

Our aim is to obtain a simple Reynolds type averaging which will transform the NSE into the NS-$\alpha$. For this purpose we will concentrate our consideration to a restricted class of fluid flows. This class, denoted by $\mathcal{P}$, will be defined (recurrently) by five assumptions, for which we will provide the respective rationale. The definition of the class $\mathcal{P}$ is inspired by the concept of regular part of the weak attractor of the 3D NSE (\cite{CR87}, \cite{CRR10}) as well as by that of the sigma weak attractor defined in \cite{BFL}.
 The new ingredient in our considerations is the hypothesis that the turbulence described by the NS-$\alpha$ is partly due to the roughness of the walls. Therefore we introduce a specifically designed mathematical model for the effect of the wall roughness onto the fluid flows, by adopting one of the Mandelbrot's paradigms \cite{M83}, which in our case is that the roughness is actually the sum of a self-similar decreasing sequence of rugosities. Our model is based on the hypothesis that the effect of each of the smaller rugosities is concentrated mainly on the flow eddies of linear size comparable with its size. We note that the simple averaging procedure we obtain does not have the generality of Holm's averaging; it can be used only if the fluid is at least partly limited by walls and the region containing the fluid has an appropriate geometry.

\section{Preliminaries}
\subsection{Mathematical backgrounds}
Throughout, we consider an incompressible viscous fluid in an immobile region $\mathcal{O}\subset \mathbb{R}^3$ subjected to a potential body force $F=-\nabla \Phi$, with a time independent potential $\Phi=\Phi(x)\in C^{\infty}(\mathcal{O})$. The velocity field of such flows,
\begin{align}
\label{nse_form}
u=u(x,t)=(u_1(x,t),u_2(x,t),u_3(x,t)), x=(x_1,x_2,x_3)\in \mathcal{O}
\end{align}
satisfies the NSE,
\begin{align}
\label{nse}
\frac{\partial}{\partial t}u+(u\cdot \nabla)u=\nu \Delta u-\nabla P, \hspace{.2 in} \nabla\cdot u=0,
\end{align}
where $P=P(x,t):=p(x,t)+\Phi(x)$, $t$ denotes the time, $\nu>0$ the kinematic viscosity, and $p=p(x,t)$ the pressure.

The NS-$\alpha$ are 
\begin{align}
\label{che}
\frac{\partial}{\partial t}v+\left(u\cdot\nabla\right)v+\sum_{j=1}^{3}v_j\nabla u_j=\nu \Delta v-\nabla Q, \hspace{.2 in} \nabla \cdot u=0,
\end{align}
where
\begin{align}
\label{vche_nse}
v=(v_1,v_2,v_3)&=(1-\alpha^2\Delta)u \nonumber \\
&=\left((1-\alpha^2\Delta)u_1,(1-\alpha^2\Delta)u_2,  (1-\alpha^2\Delta)u_3\right), 
\end{align}
and $Q$ in (\ref{che}) (like $P$ in (\ref{nse})) may depend on the time $t$.

The following boundary conditions, for both the NSE (\ref{nse}) and the NS-$\alpha$ (\ref{che}),
\begin{align}
\label{no_slip_bc}
u(x,t)=0, \hspace{.1 in}  \text{ for } x\in \partial \mathcal{O}:=\text{boundary of } \mathcal{O},
\end{align}
must hold, since the fluid we consider is viscous.

One can see that if $\alpha=0$, the NS-$\alpha$ (\ref{che}) become the NSE (\ref{nse}), so that (\ref{che}) is also referred as an $\alpha$-model of (\ref{nse}).

In the case of a channel flow, that is, $\mathcal{O}=\mathbb{R}\times\mathbb{R}\times[x_3^{(l)},x_3^{(u)}]$, where $h:=x_3^{(u)}-x_3^{(l)}>0$ is the ``height" of the channel, we recall that a vector of the form
\begin{align}
\label{vel_par_form}
(U(x_3),0,0)
\end{align}
is a $stationary$ (i.e., time independent) solution of the NSE (\ref{nse}) if and only if 
\begin{align}
\label{nse_par_form}
U(x_3)=b\left(1-\frac{(x_3-\frac{x_3^{(u)}+x_3^{(l)}}{2})^2}{(h/2)^2}\right), \hspace{.1 in} x_3\in[x_3^{(l)},x_3^{(u)}],
\end{align}
where $b$ is a constant velocity;
respectively, $(U(x_3),0,0)$ is a stationary solution of the NS-$\alpha$ (\ref{che}) if and only if,
\begin{align}
\label{vche_par_form}
U(x_3)=a_1\left(1-\frac{\cosh\left((x_3-\frac{x_3^{(u)}+x_3^{(l)}}{2})/{\alpha}\right)}{\cosh h/(2\alpha)}\right)+a_2\left(1-\frac{(x_3-\frac{x_3^{(u)}+x_3^{(l)}}{2})^2}{(h/2)^2}\right), 
\end{align}
for $x_3\in[x_3^{(l)},x_3^{(u)}]$,
where $a_1, a_2$ are constant velocities (cf. formula (9.6) in \cite{CFHO99}). Above, $\cosh(x)=\frac{e^{x}+e^{-x}}{2}$ is the hyperbolic cosine function.

To simplify our notation, we will assume up to Section \ref{section_6} that $x_3^{(l)}=0$ and $x_3^{(u)}=h$.

\subsection{Poincar$\acute{e}$ Inequality}
The following classical Poincar$\acute{e}$ inequality will be used in our discussion.
\begin{lemma}
\label{poincare}
For any $C^1$ function $\phi(y)$ defined on $[0,h]$, with $\phi(0)=\phi(h)=0$, we have
\begin{align}\label{poincare ineq}
\int_{0}^{h} (\phi^{\prime}(y))^2dy\geq \frac{1}{h^2} \int_0^h(\phi(y))^2dy.
\end{align}
\end{lemma}

\section {The class $\mathcal{P}$}
\subsection{Definition of class $\mathcal{P}$}
By definition, a function $u(x,t)$ belongs to class $ \mathcal{P}$ if it satisfies $(\bf{A}.1)-(\bf{A}.5)$,

$(\bf{A}.1)$ $u(x,t)\in C^{\infty}(\mathcal{O}\times \mathbb{R})$.

$(\bf{A}.2)$ $u(x,t)$ is periodic in $x_1$ and $x_2$, with periods $\Pi_1$ and $\Pi_2$, respectively; i.e., 
\begin{align}
\label{periodicity_u}
u(x_1+\Pi_1,x_2,x_3,t)=u(x_1,x_2,x_3,t), \hspace{.1 in} u(x_1,x_2+\Pi_2,x_3,t)=u(x_1,x_2,x_3,t).
\end{align}

\begin{remark}
\label{rk_pressure_dp}
From $(\ref{nse})$ and $(\bf{A}.2)$, we obtain, for $1\leq j \leq 3$,
\begin{align*}
\frac{\partial}{\partial x_j}\left(P(x_1+\Pi_1,x_2,x_3,t)-P(x_1,x_2,x_3,t)\right)=0
\end{align*}
and
\begin{align*}
\frac{\partial}{\partial x_j}\left(P(x_1,x_2+\Pi_2,x_3,t)-P(x_1,x_2,x_3,t)\right)=0,
\end{align*}
so $P(x_1+\Pi_1,x_2,x_3,t)-P(x_1,x_2,x_3,t)$ and $P(x_1,x_2+\Pi_2,x_3,t)-P(x_1,x_2,x_3,t)$ are functions only depending on time $t$. For simplicity, we denote
\begin{align}
\label{pressure_difference}
\left\{\begin{matrix}
p_1(t):=P(x_1+\Pi_1,x_2,x_3,t)-P(x_1,x_2,x_3,t)\\
p_2(t):=P(x_1,x_2+\Pi_2,x_3,t)-P(x_1,x_2,x_3,t)
\end{matrix}\right.
\end{align}
\end{remark}

\begin{remark}
\label{meaning of p1 and p2}
We also assume that the time independent potential function $\Phi(x)$ is periodic in $x_1$ and $x_2$ with periods $\Pi_1$ and $\Pi_2$, respectively, then, physically, $p_1(t)$ and $p_2(t)$ represent the pressure drops of the flows in $x_1$ and $x_2$ direction, respectively.
\end{remark}

$(\bf{A}.3)$ $u(x,t)$ exists for all $t\in \mathbb{R}$, and has bounded energy per mass, i.e.,
\begin{align}
\label{bdd}
 \int_{0}^{\Pi_1} \int_{0}^{\Pi_2}\int_{0}^{h}u(x,t)\cdot u(x,t) dx<\infty, \forall t\in \mathbb{R}. 
\end{align}

$(\bf{A}.4)$ there exists a constant $ 0<\bar{p}<\infty$ for which,
\begin{align*}
&0<-p_1(t)\leq \bar{p}\\
&|p_2(t)|\leq \bar{p}  
\end{align*}
for all $t \in \mathbb{R}$, where $p_1(t)$ and $p_2(t)$ are defined in (\ref{pressure_difference}).

 $(\bf{A}.5)$ $P=P(x,t)$ is bounded in $x_2$ direction, i.e.,
\begin{align*}
\sup_{x_2 \in \mathbb{R}}P(x_1,x_2,x_3,t)<\infty, \forall x_1,x_3,t\in \mathbb{R}.
\end{align*}

\begin{remark}
\label{equiv_cond}
In fact, $(\bf{A}.5)$ can be replaced with the following weaker assumption $(\bf{A}.5')$. We will provide the proof for this claim elsewhere.
\end{remark} 

$(\bf{A}.5')$ 
\begin{align}
\label{wk_A5}
\limsup_{x_2\rightarrow \pm \infty}P(x_1,x_2,x_3,t)<\infty,
\end{align}
for any given $x_1,x_3$ and $t\in \mathbb{R}$.

\begin{remark}
\label{rk_p2}
 $(\bf{A}.4)$ and $(\bf{A}.5)$ imply 
\begin{align}
\label{periodicity_p}
p_2(t)\equiv 0,
\end{align}
i.e., $P$ is periodic in $x_2$ direction. 
Indeed, for all $m\in \mathbf{Z}$,
\begin{align*}
P(x_1,x_2+m\Pi_2,x_3,t)=P(x_1,x_2,x_3,t)+mp_2(t),
\end{align*} 
one then concludes that $p_2(t)$ must equal zero by using $(\bf{A}.5)$ and letting $m\rightarrow \infty$.
\end{remark}

\subsection{Motivations and rationale for assumptions $(\bf{A}.1)$-$(\bf{A}.5)$}

In our following discussion, we will consider solutions of NSE (\ref{nse}) and/or of NS-$\alpha$ (\ref{che}) in the class $\mathcal{P}$. Before continuing, we first describe the reasons and appropriateness for making these assumptions.

Regularity property $(\bf{A}.1)$ guarantees the pointwise convergence for various Fourier series discussed in this paper. In particular, it plays an important role in the proof in Lemma \ref{simp_one}. However, this condition could be weakened using the concept of Leray-Hopf weak solutions as defined in \cite{CRR10}. 

Since the periods $\Pi_1$ and $\Pi_2$ could be taken to be arbitrarily large, assuming the periodicity in $x_1$ and $x_2$ is a reasonable approximation for experimental channel flow simulations. A technical convenience of assuming $(\bf{A}.2)$ is the availability of the Fourier series expansion.

The boundedness assumption $(\bf{A}.3)$ comes from the definition of weak global attractor of the NSE as given in \cite{CRR10}, however, we remark that even though there are several equivalent ways to define the global attractors for many dissipative systems (see \cite{T97}, \cite{CP88}), in some particular systems without having full dissipations, the appropriate notion for attractors should be defined using the boundedness (see \cite{BFL}). 

Assumptions $(\bf{A}.4)$ and $(\bf{A}.5)$ are physically reasonable, since both the pressure $P$ and the pressure drops, i.e., $p_1(t)$, $p_2(t)$, can be easily controlled by the experimenters.

\section{A simple Reynolds averaging}
\raggedbottom
Before introducing the definition of the Reynolds averaging considered in the sequel, we note the following property for the velocity fields that belong to $\mathcal{P}$:
\begin{lemma}
\label{simp_one}
If $u(x,t)=(u_1(x,t),u_2(x,t),u_3(x,t)) \in \mathcal{P}$ is a solution of the NSE $(\ref{nse})$ (or the NS-$\alpha$ $(\ref{che})$), we have 
\begin{align*}
u_3(x,t)\equiv 0, \forall x\in \Omega, t\in \mathbb{R}.
\end{align*}
\end{lemma}

\begin{proof}
By the periodicity $(\bf{A}.2)$ and no-slip boundary condition (\ref{no_slip_bc}), we can express $u$ as follows,
\begin{align*}
u(x,t)=\sum_{k\in \mathbb{Z}^2 \times \mathbb{N}}\hat{u}(t;k) E(x;k),
\end{align*}
where $k=(k_1,k_2,k_3)$, $\mathbb{N}$ denotes the set of positive integers, and
\begin{align*}
E(x;k):=e^{2\pi i (\frac{k_1x_1}{\Pi_1}+\frac{k_2x_2}{\Pi_2})}\sin(\frac{\pi k_3x_3}{h}).
\end{align*}

The incompressibility condition in (\ref{nse}) can be written as
\begin{align*}
0=\nabla \cdot u &=\sum_{k\in \mathbb{Z}^2 \times \mathbb{N}}\hat{u}_1(t;k)\frac{\partial}{\partial x_1} E(x;k)
+\sum_{k\in \mathbb{Z}^2 \times \mathbb{N}}\hat{u}_2(t;k)\frac{\partial}{\partial x_2} E(x;k)
+\frac{\partial u_3}{\partial x_3},
\end{align*}
hence,
\begin{align}
\label{form_one}
\frac{\partial u_3}{\partial x_3}=-i 2\pi \sum_{k\in \mathbb{Z}^2 \times \mathbb{N}}\left(\hat{u}_1(k)\frac{k_1}{\Pi_1}+\hat{u}_2(k)\frac{k_2}{\Pi_2}\right)\sin(\frac{\pi k_3x_3}{h}) e^{2\pi i (\frac{k_1x_1}{\Pi_1}+\frac{k_2x_2}{\Pi_2})};
\end{align}
on the other hand,
\begin{align}
\label{form_two}
\frac{\partial u_3}{\partial x_3}=\sum_{k\in \mathbb{Z}^2 \times \mathbb{N}}\hat{u}_3(k) \frac{k_3 \pi}{h}\cos(\frac{\pi k_3x_3}{h})e^{2\pi i (\frac{k_1x_1}{\Pi_1}+\frac{k_2x_2}{\Pi_2})}.
\end{align}

From (\ref{form_one}) and (\ref{form_two}), we deduce that
\begin{align}
\label{form_three}
\sum_{k_3\in \mathbb{N}}\left(\hat{u}_3(k)\frac{k_3\pi}{h}\cos(\frac{\pi k_3x_3}{h})+2\pi i [\hat{u}_1(k)\frac{k_1}{\Pi_1}+\hat{u}_2(k)\frac{k_2}{\Pi_2}] \sin(\frac{\pi k_3x_3}{h})\right)=0,
\end{align}
for all $(k_1,k_2)\in \mathbb{Z}^2$ and for all $x_3\in [0,h]$.

Now, we can rewrite (\ref{form_three}) as
\begin{align*}
& \sum_{k_3\in \mathbb{N}}\left(\frac{k_3\pi}{2h}\hat{u}_3(k)+\pi [\hat{u}_1(k)\frac{k_1}{\Pi_1}+\hat{u}_2(k)\frac{k_2}{\Pi_2}] \right)e^{i\frac{\pi k_3x_3}{h}} \\
+&
\sum_{k_3 \in \mathbb{N}}\left(\frac{k_3\pi}{2h}\hat{u}_3(k)-\pi [\hat{u}_1(k)\frac{k_1}{\Pi_1}+\hat{u}_2(k)\frac{k_2}{\Pi_2}] \right)e^{-i\frac{\pi k_3x_3}{h}}=0,
\end{align*}
from which we obtain, for all $k_3 \in \mathbb{N}$, the followings
\begin{align*}
\frac{k_3\pi}{2h}\hat{u}_3(k)+\pi [\hat{u}_1(k)\frac{k_1}{\Pi_1}+\hat{u}_2(k)\frac{k_2}{\Pi_2}]=0,
\end{align*}
and
\begin{align*}
\frac{k_3\pi}{2h}\hat{u}_3(k)-\pi [\hat{u}_1(k)\frac{k_1}{\Pi_1}+\hat{u}_2(k)\frac{k_2}{\Pi_2}]=0,
\end{align*}
so,
\begin{align}
\label{reduce_k3}
\hat{u}_3(k)k_3\equiv 0,
\end{align}
and
\begin{align*}
\hat{u}_1(k)\frac{k_1}{\Pi_1}+\hat{u}_2(k)\frac{k_2}{\Pi_2}\equiv 0,
\end{align*}
from ($\ref{reduce_k3}$), we have,
\begin{align*}
\hat{u}_3(k)\equiv 0, \forall k \in \mathbb{Z}^2\times \mathbb{N},
\end{align*}
that is, $u_3(x,t)=0$.
\end{proof}

\begin{remark}
\label{using_distribution}
It can be shown that Lemma $\ref{simp_one}$ is still valid without the regularity assumption $(\bf{A}.1)$ using the theory of distribution. This result will be reported elsewhere. 
\end{remark}

From Lemma \ref{simp_one}, the NSE (\ref{nse}) becomes
\begin{align}
\label{nse_simple}
\left\{\begin{matrix}
\frac{\partial}{\partial t}u_1+u_1 \frac{\partial}{\partial x_1}u_1+u_2\frac{\partial}{\partial x_2}u_1-\nu (\frac{\partial^2}{\partial x_1^2}+\frac{\partial^2}{\partial x_2^2}+\frac{\partial^2}{\partial x_3^2})u_1=-\frac{\partial}{\partial x_1}P\\ 
\frac{\partial}{\partial t}u_2+u_1 \frac{\partial}{\partial x_1}u_2+u_2\frac{\partial}{\partial x_2}u_2-\nu (\frac{\partial^2}{\partial x_1^2}+\frac{\partial^2}{\partial x_2^2}+\frac{\partial^2}{\partial x_3^2})u_2=-\frac{\partial}{\partial x_2}P\\
0=-\frac{\partial}{\partial x_3}P\\
\frac{\partial}{\partial x_1}u_1+\frac{\partial}{\partial x_2}u_2=0.
\end{matrix}\right.
\end{align}

The Reynolds type averaging with which we will work throughout is given in the following definition:
\begin{definition}
\label{Reynolds type average}
For any given scalar/vector function $\phi=\phi(x)$,
\begin{align}
\label{avg_q}
<\phi>(x_3):=\frac{1}{\Pi_1 \Pi_2} \int_{0}^{\Pi_1} \int_{0}^{\Pi_2} \phi dx_2dx_1.
\end{align}
\end{definition}

Applying the operation $<\cdot>$ to the first and second equations in (\ref{nse_simple}) , one gets the following Reynolds type equations
\begin{align}
\label{avg_nse}
\frac{\partial}{\partial t}
\begin{pmatrix}
<u_1(t)>\\
<u_2(t)>
\end{pmatrix}-\nu \frac{\partial^2}{\partial x_3^2} \begin{pmatrix}
<u_1(t)>\\
<u_2(t)>
\end{pmatrix}=-\begin{pmatrix}
\frac{p_1(t)}{\Pi_1}\\
\frac{p_2(t)}{\Pi_2}
\end{pmatrix}.
\end{align}

Using (\ref{periodicity_p}), we can easily obtain,

\begin{proposition}
\label{zero_avg_2nd}
For all $u(x,t)\in \mathcal{P}$, we have
\begin{align}
\label{avg_zero}
<u_2(t)>(x_3)\equiv 0, \forall x_3 \in [0,h], t \in \mathbb{R}.
\end{align}

\end{proposition}

Therefore, the averaged velocity field takes the form
\begin{align}
\label{averaged_velocity}
<u(t)>(x_3)=\begin{pmatrix}
<u_1(t)>(x_3)\\
<u_2(t)>(x_3)\\
<u_3(t)>(x_3)
\end{pmatrix}=\begin{pmatrix}
<u_1(t)>(x_3)\\
0\\
0
\end{pmatrix}
\end{align}

\begin{theorem}
\label{uniqueness_avg}
For given $p_1(t)$ and $p_2(t)$, the set 
\begin{align*}
\{<u>: u\in\mathcal{P}\}
\end{align*}
is a nonzero singleton. 
\end{theorem}

\begin{proof}
Though the proof of this lemma is elementary, we still provide the details in Appendix A.
\end{proof}

By (\ref{avg_nse}), $<u_1(t)>(x_3)$ satisfies
\begin{align}
\label{first_avg_component}
\frac{\partial}{\partial t}<u_1(t)>(x_3)-\nu \frac{\partial^2}{\partial x_3^2}<u_1(t)>(x_3)=-p_1(t)/\Pi_1,
\end{align}
with boundary conditions
\begin{align}
\label{bc_first_component}
<u_1(t)>(x_3)|_{x_3=0,h}=0.
\end{align}
In order to get an explicit form of $<u_1(t)>(x_3)$, we apply the Duhamel principal in a form adapted for $(\ref{first_avg_component})$ and $(\ref{bc_first_component})$. We can obtain the following integral representation for $<u_1(t)>(x_3)$. See Appendix B for the proof.
 
\begin{proposition}
\label{repr_by_kernel}
The following relation holds for $u(x,t)\in \mathcal{P}$
\begin{align}
\label{kernel_repr}
<u_1(t)>(x_3)=\int_{-\infty}^{t} K(x_3,t-\tau)p_1(\tau)d\tau,
\end{align}
where, the kernel function $K(x,t)$ is defined by the series,
\begin{align}
\label{kernel_K}
K(x,t)=\sum_{k=1}^{\infty}\frac{2\left((-1)^k-1\right)}{\Pi_1 k \pi}e^{-\nu(\frac{\pi k}{h})^2t}\sin\frac{\pi k x}{h}.
\end{align}
\end{proposition}

\begin{lemma}
\label{kernel_prop}
The kernel function $K(x,t)$ defined by $(\ref{kernel_K})$ satisfies the following properties:
\begin{enumerate} 

\item \begin{align}
\label{int_kernel}
\int_{-\infty}^t K(x,t-\tau)d\tau=\frac{-1}{2\Pi_1 \nu}x(h-x)
\end{align}
\item \begin{align}
\label{kernel_eqn}
\frac{\partial}{\partial t}K(x,t)-\nu \frac{\partial^2}{\partial x^2}K(x,t)=0.
\end{align}
\end{enumerate}
\end{lemma}

\begin{proof}

\begin{align*}
\int_{-\infty}^{t}K(x,t-\tau)d\tau&=\int_{-\infty}^{t}\sum_{k=1}^{\infty}\frac{2((-1)^k-1)}{\Pi_1 k \pi}e^{-\nu(\frac{\pi k}{h})^2(t-\tau)}\sin(\frac{\pi k x}{h})d\tau\\
&=\sum_{k=1}^{\infty}\int_{-\infty}^{t}\frac{2((-1)^k-1)}{\Pi_1 k \pi}e^{-\nu(\frac{\pi k}{h})^2(t-\tau)}\sin(\frac{\pi k x}{h})d\tau\\
&=\sum_{k=1}^{\infty}\frac{2((-1)^k-1)}{\Pi_1 k \pi}\sin(\frac{\pi k x}{h})\int_{-\infty}^{t}e^{-\nu(\frac{\pi k}{h})^2(t-\tau)}d\tau \\
&=\sum_{k=1}^{\infty}\frac{2((-1)^k-1)}{\Pi_1 k \pi}\sin(\frac{\pi k x}{h})\frac{h^2}{\nu (\pi k)^2},
\end{align*}
thus, (\ref{int_kernel}) is obtained by comparing the above with the following expansion
\begin{align}
\label{special_expansion_1}
x(h-x)=\sum_{k=1}^{\infty}\frac{4h^2(1-(-1)^k)}{(\pi k)^3}\sin(\frac{\pi k x}{h}).
\end{align}

Furthermore,
(\ref{kernel_eqn}) follows from a direct computation using the series representation (\ref{kernel_K}) of $K(x,t)$.
\end{proof}

Using Proposition \ref{repr_by_kernel} and (\ref{int_kernel}) in Lemma \ref{kernel_prop}, we recover the classic Poiseuille flow in the particular case when $p_1(t)$ is a constant.

\begin{theorem}
\label{poiseuille_flow_sol}
If, moreover, $p_1(t)=p_{10}<0$ for some constant $p_{10}\in \mathbb{R}$, then
\begin{align}
\label{poiseuille_form}
<u_1(t)>(x_3)=\mu x_3(h-x_3),
\end{align}
where the coefficient $\mu$ is given by
\begin{align}
\label{coef_mu}
\mu=\frac{-p_{10}}{2\Pi_1 \nu}
\end{align}
\end{theorem}

\begin{remark}
\label{const_p1_P}
From Theorem $\ref{uniqueness_avg}$ and Theorem $\ref{poiseuille_flow_sol}$, we see that
\begin{align}
\label{const_P_p1}
\{<u>: u \in \mathcal{P}, p_1(t)=p_{10}\} =\{(\frac{-p_{10}}{2\Pi_1 \nu}x_3(h-x_3),0,0)\}
\end{align}
\end{remark}

From (\ref{avg_nse}), one can easily deduce the following simple mathematical connection between the NSE (\ref{nse}) and the NS-$\alpha$ (\ref{che}).

\begin{theorem}
\label{connection_2_eqn}
Let $u(x,t)\in \mathcal{P}$. Define, for any $\alpha\in\mathbb{R}$, $\alpha>0$, 
\begin{align}
\label{particular_velocity}
V(x,t)=(1-\alpha^2\frac{\partial^2}{\partial x_3^2})<u_1(t)>(x_3),
\end{align}
and,
\begin{align}
\label{particular_q}
Q(x,t)=-\frac{1}{2}\left(<u_1(t)>^2(x_3)-\alpha^2 \big(\frac{\partial}{\partial x_3} <u_1(t)>(x_3)\big)^2\right)+\frac{x_1 p_1(t)}{\Pi_1},
\end{align}
then, 
\begin{equation}
\label{nse_to_che}
\left \{\begin{matrix}
\frac{\partial}{\partial t}V(x,t)-\nu \frac{\partial^2}{\partial x_3^2}V(x,t)=-\frac{\partial}{\partial x_1}Q \\
0=-\frac{\partial}{\partial x_2} Q\\
V\frac{\partial}{\partial x_3}<u_1(t)>(x_3)=-\frac{\partial}{\partial x_3} Q
\end{matrix}\right.
\end{equation}
Equivalently, $(<u_1(t)>(x_3),0,0)$ is a solution of the NS-$\alpha$ $(\ref{che})$ with corresponding $Q(x,t)$ defined by $(\ref{particular_q})$.
\end{theorem}

\section{Transition mechanism from NSE to NS-$\alpha$}
\label{section_6}
\subsection {Motivations}
The flows we considered in this paper are driven by the pressure drop, as we only consider a potential body force, thus, the pressure term will play essential roles in the study of these equations. Moreover, from (\ref{kernel_repr}), we see that, for any $T>0$,
\begin{align*}
U_1(x_3):&=\frac{1}{T}\int_{0}^{T}<u_1(t)>(x_3)dt\\
&=\frac{1}{T}\int_{0}^{T}\int_{-\infty}^{t}K(x_3,t-\sigma)p_1(\sigma)d\sigma dt\\
&\overset{\xi=t-\sigma}{=}\int_{0}^{\infty}K(x_3,\xi)P_1(\xi)d\xi \\
&=\sum_{k=1}^{\infty}\frac{2((-1)^k-1)}{\Pi_1 k \pi}\left(\int_{0}^{\infty}e^{-\nu(\frac{k\pi}{h})^2 \xi}P_1(\xi)d\xi \right)\sin\frac{k\pi x_3}{h},
\end{align*}
where
\begin{align*}
 P_1(\xi)=\frac{1}{T}\int_{0}^{T}p_1(t-\xi)dt.
\end{align*}

It follows from the Plancherel's theorem that 
\begin{align*}
||U_1||^2_{L^2([0,h])}&=\sum_{k=1}^{\infty} \frac{4((-1)^k-1)^2}{(\Pi_1 k \pi)^2}\left(\int_{0}^{\infty}e^{-\nu(\frac{k\pi}{h})^2\xi}P_1(\xi)d\xi \right)^2||\sin\frac{k \pi x_3}{h}||^2_{L^2([0,h])}\\
&=\sum_{k=1}^{\infty} \frac{2h((-1)^k-1)^2}{(\Pi_1 k \pi)^2}\left(\int_{0}^{\infty}e^{-\nu(\frac{k\pi}{h})^2 \xi}P_1(\xi)d\xi\right)^2\\
&\leq \left( \text{use the bound for } p_1(t) \text{ in } (\bf{A}.4)\right)\\
&\leq \bar{p}^2 \sum_{k=1}^{\infty}\frac{2h^5((-1)^k-1)^2}{\Pi_1^2 \nu^2 {k\pi}^6}\\
&\leq \bar{p}^2 \frac{2h^5}{\Pi_1^2 \nu^2 (\pi)^6} \sum_{k=1}^{\infty}\frac{4}{(2k-1)^2}.
\end{align*}

Using the identity
\begin{align}
\label{a_series_sum}
\sum_{k=1}^{\infty}\frac{1}{(2k-1)^2}=\frac{\pi^2}{8},
\end{align}
we obtain
\begin{proposition}
\label{sel}
\begin{align}
\label{relation_Re_p}
Re\leq\frac{\bar{p}h^{3}}{\Pi_1 \nu^2 \pi^2},
\end{align}
where,
\begin{align}
\label{Reynolds}
Re:=\frac{h^{1/2}||U_1||_{L^2([0,h])}}{\nu}
\end{align}
is the Reynolds number related to the flow.
\end{proposition}

From (\ref{relation_Re_p}) in Proposition \ref{sel} we see that the magnitude of the pressure drop forms an upper estimate of that of the velocity in the channel, thus if the magnitude of flows velocity become significant, or equivalently, the Reynolds number $Re$ of the flow is large, then the solution of the velocity from the NSE will no longer satisfy this upper estimate. At this moment, the fluid will ``select" the NS-$\alpha$ model instead of the NSE. This motivates our considerations in this section.

The Prandtl's wall roughness idea suggests our conjecture for the transition from the NSE to the NS-$\alpha$. That is, the roughness of the wall in $x_3$ direction may ``introduce'' the operator $(1-\alpha^2 \Delta)$. 

To model the effect of the wall roughness onto the fluid flow, it is necessary to consider the general channel geometry $\mathcal{O}=\mathbb{R}\times\mathbb{R}\times[x_3^{(l)},x_3^{(u)}]$, since we have to consider the effects of both the upper and lower walls. However, we stress that, in order to keep the symmetry property of the fluid flow, it is reasonable to assume that the wall roughness also satisfies the appropriate symmetry property such that the center of the channel does not move after taking into account of the change of vertical distance of the channel due to the wall rugosities, and consequently, only the change of the height matters. Therefore, without loss of generality, it suffices to consider the particular case where $x_3^{(l)}=0$ and $x_3^{(u)}=h$.

\subsection {Another property of the kernel $K(x,t;h)$}
We consider the solution $<u_1(t)>(x_3)$, represented in (\ref{kernel_repr}) to be also a function of the wall height $h$, that is,
\begin{align}
\label{u_1_in_h}
<u_1(t)>(x_3;h)=\int_{-\infty}^{t} K(x_3,t-\tau;h)p_1(\tau)d\tau,
\end{align}
where
\begin{align}
\label{kernel_K_in_h}
K(x,t;h)=\sum_{k=1}^{\infty}\frac{2\left((-1)^k-1\right)}{\Pi_1 k \pi}e^{-\nu(\frac{\pi k}{h})^2t}\sin\frac{\pi k x}{h}.
\end{align}

\begin{lemma}
\label{kernel_prop_more}
The kernel $K(x,t;h)$ also satisfies,
\begin{align}
\label{more_on_K}
\frac{\partial K}{\partial h}=-\frac{x}{h}\frac{\partial K}{\partial x}-\frac{2t}{h}\frac{\partial K}{\partial t}.
\end{align}
\end{lemma}

\subsection {Modification of $p_1(t)$ due to the roughness of the channel walls}

We introduce the effect of the wall roughness to update the pressure drop $p_1(t)$ by considering it to be a function of the roughness $\mathfrak{R}=\mathfrak{R}(x_1,x_2)$, which is only a function of the variables $x_1$ and $x_2$.

Motivated by (\ref{first_avg_component}), we replace $p_1(t)$ by the following expression,
\begin{align}
\label{new_p1_def}
p_1(t;h+\mathfrak{R}):=\Pi_1\left(-\frac{\partial}{\partial t}<u_1(t)>(x_3;h+\mathfrak{R})+\nu\frac{\partial^2}{\partial x_3^2}<u_1(t)>(x_3;h+\mathfrak{R})\right),
\end{align}
where, recall the Reynolds type average $<\cdot>=\frac{1}{\Pi_1\Pi_2}\int_{0}^{\Pi_1}\int_{0}^{\Pi_2}\cdot dx_2dx_1$, defined in (\ref{avg_q}).

Now, we use the first order linear approximation, that is, we approximate $p_1(t;h+\mathfrak{R})$ by
\begin{align}
\label{new_p1_linear}
\Pi_1\left((-\frac{\partial }{\partial t}+\nu \frac{\partial ^2}{\partial x_3^2})<u_1(t)>(x_3;h)+\mathfrak{R}(x_1,x_2)(-\frac{\partial }{\partial t}+\nu \frac{\partial ^2}{\partial x_3^2})\frac{\partial}{\partial h} <u_1(t)>(x_3;h)\right).
\end{align}
By the identity (\ref{more_on_K}) in Lemma \ref{kernel_prop_more}, (\ref{new_p1_linear}) equals
\begin{align*}
p_1(t)+\Pi_1\mathfrak{R}(x_1,x_2)(-\frac{\partial }{\partial t}+\nu \frac{\partial ^2}{\partial x_3^2})\int_{-\infty}^{t}\left((-\frac{x_3}{h})\frac{\partial K}{\partial x_3}(x_3,t-\tau;h)-\frac{2(t-\tau)}{h}\frac{\partial K}{\partial t}(x_3,t-\tau;h)\right)p_1(\tau)d\tau,
\end{align*}
which, after using (\ref{kernel_eqn}) in Lemma \ref{kernel_prop}, can be simplified to be
\begin{align}
\label{temp1_new_p1}
p_1(t)\left(1+\mathfrak{R}(x_1,x_2)\Pi_1\frac{x_3}{h}\frac{\partial K}{\partial x_3}(x_3,0;h)\right).
\end{align}

However, the expression in (\ref{temp1_new_p1}) depends on $x_3$, thus, we take average in $x_3$ in (\ref{temp1_new_p1}) to get the form of updated $p_1(t)$, namely,
\begin{align}
\label{new_form_p1}
p_1^{new}(t)&:=\frac{1}{h}\int_{0}^{h} p_1(t) \left(1+\Pi_1\mathfrak{R}(x_1,x_2)\frac{x_3}{h}\frac{\partial K}{\partial x_3}(x_3,0;h)\right)dx_3\\ \nonumber
&=p_1(t)\left(1+\frac{\mathfrak{R}}{h}\sum_{k=1}^{\infty} \frac{2\left((-1)^k-1\right)^2}{(k\pi)^2}\right)\\ \nonumber
&=( \text{use } (\ref{a_series_sum})) \\ \nonumber
&= p_1(t)\left(1+\frac{\mathfrak{R}}{h}\right). \nonumber
\end{align}

\subsection{Update of $<u_1(t)>(x_3;h)$}
Replacing $p(t)$ in the expression (\ref{u_1_in_h}) by the updated form (\ref{new_form_p1}), we get the following,
\begin{align}
\label{temp_u1}
&\int_{-\infty}^{t}K(x_3,t-\tau;h)p_1^{new}(\tau)d\tau\\\nonumber
&=\int_{-\infty}^{t}K(x_3,t-\tau;h)p_1(\tau)d\tau+\int_{-\infty}^{t}K(x_3,t-\tau;h)\frac{\mathfrak{R}}{h}p_1(\tau)d\tau \\ \nonumber
&=<u_1(t)>(x_3;h)+\frac{1}{h}\int_{-\infty}^{t}\sum_{k=1}^{\infty}\frac{2((-1)^k-1)}{\Pi_1 k\pi}e^{-\nu(\frac{k\pi}{h})^2(t-\tau)}\mathfrak{R}(x_1,x_2)\left(\sin{\frac{k\pi x_3}{h}}\right)p_1(\tau)d\tau,
\end{align}

In order to update $<u_1(t)>(x_3;h)$, we need a mathematical description for the roughness. Our mathematical definition of the wall roughness is an application of  Mandelbrot's paradigm(see \cite{MAN2}) that the roughness of a wall is produced by a sum of small decreasing rugosities of the walls which are assumed to be connected by adequate self-similarities.

For this purpose, we first introduce the following definitions:
for $j=1,2$, let $r_j:\mathbb{R} \rightarrow \mathbb{R}$ denote a function satisfying
\begin{align*}
r_j (x)=r_j (x+\pi_j), \forall x \in \mathbb{R}
\end{align*}
where
\begin{align*}
\pi_j=\frac{\Pi_j}{N_j}, \hspace{.2 in} r_j (x)=\left\{\begin{matrix}
r_j (0)\hspace{.2 in} |x|<\delta_j \\ 
0\hspace{.2 in}\delta_j \leq |x|<\frac{\pi_j}{2} 
\end{matrix}\right.
\end{align*}
for some $N_j \in \mathbb{N}, N_j>0$  $(j=1,2)$, and $\delta_1, \delta_2 $ are the length of the edges in the $x_1$ and $x_2$ directions, respectively, of the rectangular parallelepiped we will define later.

In fact, we note that the periods $\pi_j (j=1,2)$ are properties of the walls and it is plausible to consider that our previously defined periods $\Pi_j$ are connected to the $\pi_j 's, j=1,2$.

The progenitor of the roughness system is
\begin{align*}
rug_1= \{ x=[x_1,x_2,x_3]: 0 \leq x_3 \leq \frac{r_1 (x_1) r_2 (x_2)}{h} \}
\end{align*}
Note that $rug_1$ is a system of rectangular parallelepipeds with volume 
\begin{align*}
vol_1=\frac{4 \delta_1 \delta_2 r_1 (0) r_2 (0)}{h}.
\end{align*}

The descendant generations of the rugosities $rug_n (n=2,3, ...)$ are
\begin{align*}
rug_n= \{ x=[x_1,x_2,x_3]: 0 \leq x_3 \leq \frac{r_1 (n x_1) r_2 (n x_2)}{ n^2 h} \}
\end{align*}
with volume
\begin{align*}
vol_n=\frac{1}{n^4} vol_1.
\end{align*}

Since the roughness is the superposition of different rugosities that bear similarities, mathematically we have
\begin{align}
\label{rf_form}
\mathfrak{R}(x_1,x_2)=\sum_{n=1}^{\infty}  \mathfrak{R}_{(n,k)}(x_1,x_2),
\end{align}
where
\begin{align}
\label{rug_form}
\mathfrak{R}_{(n,k)}(x_1,x_2)=e(n) s(n,k) \frac{r_1 (n x_1) r_2 (n x_2)}{n^2 h},
\end{align}
where $e(n)$ represents the effect of the rugosity $rug_n$ onto the fluid and $s(n,k)$ represents how the rugosities ``pick" the wavenumber. When the fluid flows through the wall, it will have the strongest effect to the smallest rectangular parallelepiped. So the smaller the volume is, the stronger the fluid effects. It is reasonable to assume $e(n)=\frac{c_1}{vol_n}=\frac{c_1 n^4}{vol_1}$, where $c_1$ is a dimensional constant.

As mentioned in the introduction, assuming that the rugosity $rug_n$ will only affect the wave numbers having size comparable with its own size, the rugosity will ``see" the size of the wavenumber of the fluid field, and ``pick" its favorable wavenumber to interact.  Also from (\ref{u_1_in_h}) and (\ref{kernel_K_in_h}) we notice that $k$ needs to be odd. Similarly, we need $n$ to be odd. So we can assume $s(n,k)=\chi_{(\frac{h-\epsilon_{n} h}{n}, \frac{h-\epsilon_{n-1} h}{n-1}]} (\frac{h}{k}) \times \chi_{2 \mathbb{Z}+1}(n)$, where $\epsilon_{n}=h^{-1}{h_1 \sum_{l=1}^n \frac{1}{l^2}}$, $h_1$ is a fixed small constant, and $h$ is sufficiently small.

\begin{proposition}
\label{pick1}
For each fixed odd number $k$, $\sum_{n=1}^{\infty}  \mathfrak{R}_{(n,k)}(x_1,x_2)=\mathfrak{R}_{(k,k)}(x_1,x_2).$
\end{proposition}

\begin{proof}
See Appendix C.
\end{proof}

Then the second term of the RHS of (\ref{temp_u1}) becomes
\begin{align*}
&\frac{1}{h}\int_{-\infty}^{t}\sum_{k=1}^{\infty}\frac{2((-1)^k-1)}{\Pi_1 k\pi}e^{-\nu(\frac{k\pi}{h})^2(t-\tau)}\mathfrak{R}(x_1,x_2)(\sin{\frac{k\pi x_3}{h}})p_1(\tau)d\tau \\
&=\frac{1}{h}\int_{-\infty}^{t}\sum_{k=1}^{\infty}\left(\frac{2((-1)^k-1)}{\Pi_1 k\pi}e^{-\nu(\frac{k\pi}{h})^2(t-\tau)}\sum_{n=1}^{\infty}\mathfrak{R}_{(n,k)}(x_1,x_2)\sin{\frac{k\pi x_3}{h}}\right)p_1(\tau)d\tau \\
&=\frac{1}{h}\int_{-\infty}^{t}\sum_{k=1}^{\infty}\left(\frac{2((-1)^k-1)}{\Pi_1 k\pi}e^{-\nu(\frac{k\pi}{h})^2(t-\tau)}\mathfrak{R}_{(k,k)}(x_1,x_2)\sin{\frac{k\pi x_3}{h}}\right)p_1(\tau)d\tau \\
&=\frac{1}{h}\int_{-\infty}^{t}\sum_{k=1}^{\infty}\frac{2((-1)^k-1)}{\Pi_1 k\pi}e^{-\nu(\frac{k\pi}{h})^2(t-\tau)} \frac{c_1 k^4}{vol_1} \frac{r_1 (k x_1) r_2(k x_2)}{k^2 h}\sin{\frac{k\pi x_3}{h}}p_1(\tau)d\tau \\
&=\frac{1}{h}\int_{-\infty}^{t}\sum_{k=1}^{\infty}\frac{2((-1)^k-1)}{\Pi_1 k\pi}e^{-\nu(\frac{k\pi}{h})^2(t-\tau)} \frac{c_1 k^2}{vol_1 h} r_1 (k x_1) r_2(k x_2)\sin{\frac{k\pi x_3}{h}}p_1(\tau)d\tau
\end{align*}
averaging with respect to $x_1$ and $x_2$, we have
\begin{align*}
\frac{1}{h}\int_{-\infty}^{t}\sum_{k=1}^{\infty}\frac{2((-1)^k-1)}{\Pi_1 k\pi}e^{-\nu(\frac{k\pi}{h})^2(t-\tau)}\frac{c_1 k^2}{vol_1 h}r_1 (0) r_2 (0)\sin{\frac{k\pi x_3}{h}}p_1(\tau)d\tau,
\end{align*}
thus, the updated averaged velocity after taking into account the wall roughness is,
\begin{align}
\label{new_u_1}
<u_1(t)>^{new}(x_3;h)&=<u_1(t)>(x_3;h)-\frac{c_1 r_1 (0) r_2 (0)}{\pi^2 vol_1} \Delta <u_1(t)>(x_3;h), 
\end{align}
whence the Laplacian operator arises naturally. We obtain the NS-$\alpha$ with $\alpha=\sqrt{\frac{c_1 r_1 (0) r_2 (0)}{\pi^2 vol_1}}=\sqrt{\frac{c_1 h}{4\pi^2 \delta_1\delta_2}}$:
\begin{align}
\label{new_u_2}
<u_1(t)>^{new}(x_3;h)&=(1-{\alpha}^2 \Delta)<u_1(t)>(x_3;h).
\end{align}


\section{Appendices}

\subsection{Appendix A: Proof of Theorem \ref{uniqueness_avg} }
\begin{proof}
Let $u$ and $v$ be any two elements in $\mathcal{P}$ which are the solutions of the NSE (\ref{nse}).
Denoting
\begin{align*}
w(x_3,t)=(w_1(x_3,t),w_2(x_3,t),w_3(x_3,t)):=<u(t)>(x_3)-<v(t)>(x_3),
\end{align*}
then
\begin{align*}
w_2(x_3,t)=w_3(x_3,t)=0,
\end{align*}
for all $x_3 \in [0,h]$, $t\in \mathbb{R}$.

Moreover, from (\ref{avg_nse}), 
\begin{align*}
\frac{\partial}{\partial t}w_1(x_3,t)-\nu \frac{\partial^2}{\partial x_3^2}w_1(x_3,t)=0,
\end{align*}
hence,
\begin{align*}
\frac{1}{2}\frac{d}{dt}\int_{0}^{h}w_1^2(x_3,t)dx_3+\nu \int_{0}^h|\frac{\partial w_1}{\partial x_3}(x_3,t)|^2dx_3=0.
\end{align*}
Notice that 
\begin{align*}
w_1(x_3,t)|_{x_3=0,h}=0,
\end{align*}
then Poincar$\acute{e}$ inequality (\ref{poincare ineq}) is applicable, so we obtain,
\begin{align*}
\frac{1}{2}\frac{d}{dt}\int_{0}^{h}w_1^2(x_3,t)dx_3+\frac{\nu}{h^2} \int_{0}^hw_1^2(x_3,t)dx_3\leq 0,
\end{align*}
therefore, for any $t_0<t$, 
\begin{align*}
0\leq \int_0^hw_1^2(x_3,t)dx_3 \leq e^{-\frac{2}{h^2}(t-t_0)}\int_0^h w_1^2(x_3,t_0)dx_3 ,
\end{align*}
the uniqueness follows from using assumption $(\bf{A}.3)$ and letting $t_0\rightarrow -\infty$ in the last inequality.

Finally, if this singleton is zero, then from (\ref{avg_nse}), $p_1(t)$ must be zero, which contradicts with the assumption $(\bf{A.}4)$ that $p_1(t)$ is never zero.
\end{proof}

\subsection{Appendix B: Proof of Proposition \ref{repr_by_kernel}}
\begin{proof}
It is well known that (see \cite{N65}) an orthonormal basis for 
\begin{align}
\label{domain_A}
\mathcal{D}(A):=\{\phi(x) \in C^2([0,h]): \phi(x)|_{x=0,h}=0\},
\end{align}
where $A:=-\frac{\partial^2}{\partial x^2}$, is 
\begin{align*}
\begin{Bmatrix} \sqrt{2/h}\sin(\pi k x/h) \end{Bmatrix}_{k=1}^{\infty}.
\end{align*}

Therefore, $<u_1(t)>(x_3)$ can be expanded as,
\begin{align*}
<u_1(t)>(x_3)=\sum_{k=1}^{\infty} <u_1(t)>^{\widehat{}}(k) \sqrt{\frac{2}{h}}\sin\frac{\pi k x_3}{h}.
\end{align*}
To obtain an explicit form for $<u_1(t)>^{\widehat{}}(k)$, the coefficients in the Fourier sine series expansion, we introduce the Fourier expansion  in equation (\ref{first_avg_component}) to get
\begin{align*}
\frac{\partial}{\partial t} <u_1(t)>^{\widehat{}}(k)+\nu (\frac{\pi k}{h})^2<u_1(t)>^{\widehat{}}(k)&=\int_0^h -\frac{p_1(t)}{\Pi_1}\sqrt{\frac{2}{h}}\sin \frac{\pi k x_3}{h} dx_3\\
&=-\sqrt{\frac{2}{h}}\frac{p_1(t)}{\Pi_1}\frac{h}{\pi k}\bigg(1-(-1)^k\bigg),
\end{align*}
whence, for $t_0<t$, 
\begin{align*}
<u_1(t)>^{\widehat{}}(k)=e^{-\nu (\frac{\pi k}{h})^2(t-t_0)}<u_1(t_0)>^{\widehat{}}(k)
+\sqrt{\frac{2}{h}}\frac{h}{\Pi_1\pi k}\bigg((-1)^k-1\bigg) \int_{t_0}^t e^{-\nu (\frac{\pi k}{h})^2(t-\tau)}p_1(\tau)d\tau.
\end{align*}
Letting $t_0 \rightarrow -\infty$ and using $(\bf{A}.3)$, we have
\begin{align*}
<u_1(t)>^{\widehat{}}(k)=\sqrt{\frac{2}{h}}\frac{h}{\Pi_1\pi k}\bigg((-1)^k-1\bigg) \int_{-\infty}^t e^{-\nu (\frac{\pi k}{h})^2(t-\tau)}p_1(\tau)d\tau,
\end{align*}
and the result follows.
\end{proof}

\subsection{Appendix C: a ``matching" argument}
\begin{proof}[Proof of Proposition \ref{pick1}]
For any odd number $k$, we have
\begin{align*}
\sum_{n=1}^{\infty}  \mathfrak{R}_{(n,k)}(x_1,x_2)=\sum_{n=1}^{\infty} e(n) s(n,k) \frac{r_1 (n x_1) r_2 (n x_2)}{n^2 h},
\end{align*}
where $s(n,k) \neq 0$ only if $n$ is odd and $\frac{h}{k} \in (\frac{h-\epsilon_{n} h}{n}, \frac{h-\epsilon_{n-1} h}{n-1}]$.

Now,
\begin{align*}
\frac{h}{k} \in (\frac{h-\epsilon_{n} h}{n}, \frac{h-\epsilon_{n-1} h}{n-1}] \Rightarrow n \in (k(1-\epsilon_n), k(1-\epsilon_{n-1})+1]
\end{align*}
Since $h_1<<h$, $\epsilon_1$ is very small, we have
\begin{align*}
1<[k(1-\epsilon_{n-1})+1]-[k(1-\epsilon_n) ]=1+\frac{k \epsilon_1}{n^2}<2,
\end{align*}
this implies that there are at most two consecutive integers locate in between the interval $(k(1-\epsilon_n), k(1-\epsilon_{n-1})+1]$. The only two possibilities are $k$ and $k-1$, but $n$ need to be odd, so we have $n=k$. 
\end{proof}

\section{Acknowledgement}
This work was supported in part by NSF grants number DMS-1109638 and DMS-1109784.



\begin{bibdiv}
\begin{biblist}

\bib{BFL}{article}{
title={On the attractor for the semi-Dissipative Boussinesq Equations},
author={Biswas, A.},
author={Foias, C.},
author={Larios, A.}
journal={Submitted}
}


\bib{CFHO98}{article}{
  title={The Camassa-Holm equations as a closure model for turbulent channel and pipe flow},
  author={Chen, S. Y.},
 author={ Foias, C.},
 author={ Holm, D. D.},
author={ Olson, E.},
author={Titi, E. S.},
 author={ Wynne, S.},
  journal={Physical Review Letters},
  volume={81},
  number={24},
  pages={5338-5341},
  year={1998},
  publisher={APS}
}

\bib{CFHO99}{article}{
title={A connection between the Camassa--Holm equations and turbulent flows in channels and pipes},
  author={Chen, S. Y.},
 author={ Foias, C.},
 author={ Holm, D. D.},
 author={ Olson, E.},
author={ Titi, E. S.},
 author={ Wynne, S.},
  journal={Physics of Fluids},
  volume={11},
  number={8},
  pages={2343--2353},
  year={1999},
  publisher={AIP Publishing}
}

\bib{CFHOTW99}{article}{
  title={The Camassa--Holm equations and turbulence},
 author={Chen, S. Y.},
 author={ Foias, C.},
 author={ Holm, D. D.},
 author={ Olson, E.},
author={ Titi, E. S.},
 author={ Wynne, S.},
  journal={Physica D: Nonlinear Phenomena},
  volume={133},
  number={1},
  pages={49--65},
  year={1999},
  publisher={Elsevier}
}

\bib{A02}{article}{
  title={Turbulent boundary layer equations},
  author={Cheskidov, A.},
  journal={Comptes Rendus Mathematique},
  volume={334},
  number={5},
  pages={423--427},
  year={2002},
  publisher={Elsevier}
}

\bib{A04}{article}{
  title={Boundary layer for the Navier-Stokes-alpha model of fluid turbulence},
  author={Cheskidov, A.},
  journal={Archive for rational mechanics and analysis},
  volume={172},
  number={3},
  pages={333--362},
  year={2004},
  publisher={Springer}
}





\bib{CDT02}{article}{
  title={The three dimensional viscous Camassa--Holm equations, and their relation to the Navier--Stokes equations and turbulence theory},
  author={Foias, C.},
 author={Holm, D. D.},
 author={Titi, E. S.},
  journal={Journal of Dynamics and Differential Equations},
  volume={14},
  number={1},
  pages={1--35},
  year={2002},
  publisher={Springer}
}

\bib{CP88}{book}{
title={Navier-stokes equations},
  author={Constantin, P.},
author={Foias, C.},
  year={1988},
  publisher={University of Chicago Press}
}

\bib{CR87}{article}{
title={The connection between the Navier-Stokes equations, dynamical systems, and turbulence theory},
  author={Foias, C.},
 author={Temam, R.},
  journal={in``Directions in Partial Differential Equations" (Madison, WI, 1985), Publ. Math. Res. Center Univ. Wisconsin, 54, Academic Press, Boston, MA},
  pages={55--73},
  year={1987}
}
\bib{CRR10}{article}{
  title={Topological properties of the weak global attractor of the three-dimensional Navier-Stokes equations},
  author={Foias, C.},
 author={ Rosa, R.},
 author={ Temam, R.},
  journal={Discrete Contin. Dyn. Syst},
  volume={27},
  number={4},
  pages={1611--1631},
  year={2010}
}


\bib{H99}{article}{
author={D. D. Holm},
title={Fluctuation effects on 3D Lagrangian mean and Eulerian mean fluid motion},
journal={Physica D},
volume={133},
year={1999},
pages={215--269}
}





\bib{M83}{book}{
  title={The fractal geometry of nature},
  author={Mandelbrot, B. B.},
  volume={173},
  year={1983},
  publisher={Macmillan}
}

\bib{MAN2}{book}{
  title={Multifractals and $1/f$ noise},
  author={Mandelbrot, B. B.},
  volume={selected volume},
  year={1999},
  publisher={Springer}
}


\bib{N65}{book}{
  title={Introduction to real functions and orthogonal expansions},
  author={B. S. Nagy},
  year={1965},
  publisher={Akad{\'e}miai Kiad{\'o}}
}

\bib{T97}{book}{
title={Infinite dimensonal dynamical systems in mechanics and physics},
author={Temam, R.},
year={1997},
publisher={Springer Science  \& Business Media},
volume={68}
}

\end{biblist}
\end{bibdiv}
\end{document}